\documentclass[a4paper]{article}
\usepackage{amsmath,amsthm,amssymb}
\usepackage[hyphens]{url} 
\usepackage[dvips]{graphicx}
\theoremstyle{plain}

\theoremstyle{definition}
 \newtheorem{exm}{Example}[section]
 
\theoremstyle{remark}
 
 \numberwithin{equation}{section}

\renewcommand{\le}{\leqslant}\renewcommand{\ge}{\geqslant}
\renewcommand{\leq}{\leqslant}

\title{The Syntax of Polytopal Projections: From Permutohedra to Associahedra}

\author{
{Miloš Adžić} \\ {\small Faculty of Philosophy}\\[-2mm] {\small University of Belgrade}\\[-2mm] {\small Belgrade, Serbia}
\and
{Filip D. Jevti\'{c}} \\ {\small Matematical Institute}\\[-2mm] {\small SASA}\\[-2mm] {\small Belgrade, Serbia}\\[-2mm] 
}

\date{}

\newcommand{\Sym}{\mathfrak{S}}      
\newcommand{\Y}{\mathbb{Y}}          
\newcommand{\Prm}{\mathcal{P}}       
\newcommand{\Asc}{\mathcal{K}}       
              
\newcommand{\std}{\mathrm{std}}      
\newcommand{\eval}{\varepsilon}
\newtheorem{theorem}{Theorem}[section]
\newtheorem{proposition}[theorem]{Proposition}
\newtheorem{lemma}[theorem]{Lemma}
\newtheorem{definition}[theorem]{Definition}
\newtheorem{remark}[theorem]{Remark}
\newtheorem{corollary}[theorem]{Corollary}
\usepackage{geometry}
\usepackage{graphicx}
\usepackage{tikz}
\usepackage{subcaption}
\usepackage{mathtools}
\usepackage{xcolor}

\begin{document}

\maketitle

\begin{abstract}
Tonks' projection from the permutohedron to the associahedron and the
Loday--Ronco map both send permutations to planar binary trees. We give a
syntactic account of these maps in the equational calculus of the free
non-symmetric, non-unital operad on one binary generator. The vertex restriction
of Tonks' projection is obtained by evaluating the head-insertion encoding on
the reversed permutation, while the Loday--Ronco map is obtained by evaluating
the decreasing encoding. We also give a local operadic proof that Tonks' vertex
map is order-preserving from the weak Bruhat order to the Tamari order.
\end{abstract}

\section{Introduction}

The associahedron and the permutohedron are among the central objects of the theory of polytopes, and beyond. In this paper we consider the associahedron \(\Asc^{n+1}\) and the permutohedron \(\Prm^n\), which in our indexing convention both have dimension \(n-1\). They belong to the broader family of nestohedra, equivalently hypergraph polytopes (see~\cite{DP2}). A classical connection between them is given by Tonks'~\cite{Ton} cellular quotient map
\[
\theta:\Prm^n\longrightarrow \Asc^{n+1}.
\]
Its restriction to vertices yields a map \(\phi\) from permutations to (planar) binary trees. A closely related map \(\psi\), again from permutations to binary trees, was introduced by Loday--Ronco~\cite{LoR} in their study of dendriform algebras:
\[
\psi:\Sym_n\longrightarrow \Y_n
\]

\begin{figure}[htbp]
    \centering
    \includegraphics[width=0.92\textwidth]{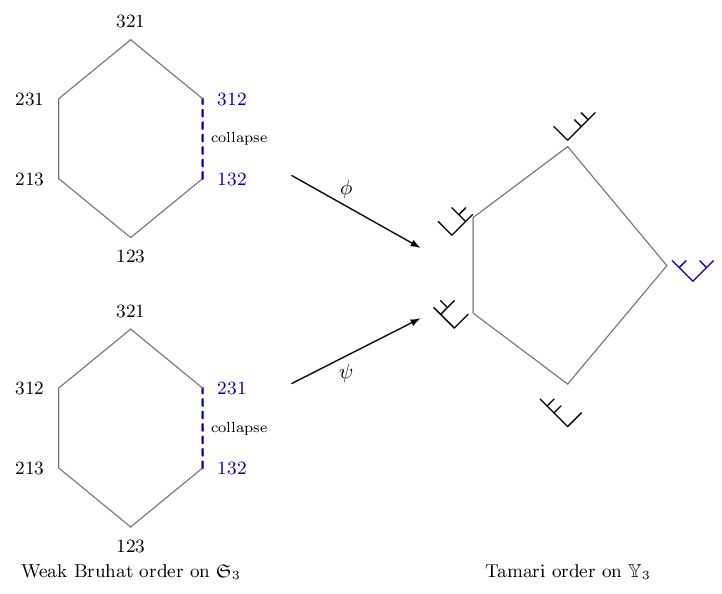}
    \caption{The upper hexagon is the right and the lower hexagon is the left weak Bruhat order on \(\Sym_3\). Tonks' map \(\phi\) collapses the edge \(132\text{--}312\), while the Loday--Ronco map \(\psi\) collapses \(132\text{--}231\) to the same binary tree. In both cases, the quotient is the Tamari order on \(\Y_3\).}\label{fig:projection}
\end{figure}

Figure~\ref{fig:projection} shows the smallest nontrivial case of Tonks' and Loday--Ronco's maps at the level of vertices. We describe both \(\phi\) and \(\psi\) in
the equational calculus \(\mathcal I\) of Do\v{s}en and Petri\'c~\cite{DP1},
which axiomatises the free non-symmetric, non-unital operad on one binary
generator. The two maps arise from indexed operadic terms:
head-insertion for Tonks' map and decreasing insertion for the Loday--Ronco
map. The same syntax also gives a local proof of order preservation for
Tonks' map.

\section{Preliminaries}

In this section we give a rigorous definition of Tonks' cellular quotient map $\theta:\Prm^n\longrightarrow \Asc^{n+1}$, and Loday--Ronco map $\psi:\Sym_n\longrightarrow \Y_n$. To that end, we set up some conventions and recall some facts of combinatorics.

Let $[0]=\emptyset$ and $[n]=\{1,\dots,n\}$ for $n\ge 1$. 
Let $\Sym_n$ denote the symmetric group on $[n]$, with its elements given in one-line notation; in
particular, $\Sym_0=\{\emptyset\}$. 
For $n>0$, let $\Y_n$ denote the set of binary trees with $n$ internal nodes, equivalently
with $n+1$ leaves, and let \(\Y_0=\{\ast\}\), where \(\ast\) denotes the
trivial tree with a single vertex.
Binary trees can be grafted: for \(T_1\in\Y_p\) and \(T_2\in\Y_q\), let  \(T_1\vee T_2\in\Y_{p+q+1}\) be the tree obtained by attaching \(T_1\) and \(T_2\) as the left and right subtrees of a new root, respectively.

Throughout the paper we use the convention that the $(n-1)$-dimensional
permutohedron \(\Prm^n\) has vertex set \(\Sym_n\), while the
\((n-1)\)-dimensional associahedron \(\Asc^{n+1}\) has vertex set \(\Y_n\).
Thus the shift in the notation for \(\Asc^{n+1}\) reflects the indexing of
binary trees by their number of leaves.

For a finite word \(a=a_1\cdots a_k\) of distinct integers, its \emph{standardization}
\(\std(a)\in\Sym_k\) is obtained by replacing the smallest letter of \(a\) by \(1\), the
second smallest by \(2\), and so on, while preserving relative order. If \(a=\emptyset\), we set \(\std(\emptyset)=\emptyset\in\Sym_0\). For a word \(a=a_1\cdots a_k\), let \(w(a)=a_k\cdots a_1\) denote its \emph{reversal}. For \(1\leq t\leq k\), write \(a_{\leq t}=a_1\cdots a_t\) for a prefix of the word \(a\) of length \(t\).

Given a word \(a=a_1\cdots a_k\) of distinct integers and an integer \(x\),
let \(a^{<x}\) denote the subsequence of entries of \(a\) that are strictly
less than \(x\), and let \(a^{>x}\) denote the subsequence of entries of \(a\)
that are strictly greater than \(x\). These subsequences are taken in the
original left-to-right order and may be empty.

For \(\pi=\pi_1\cdots\pi_n\in\Sym_n\), an \emph{inversion} of \(\pi\) is a pair
\((i,j)\) with \(1\le i<j\le n\) and \(\pi_i>\pi_j\). We write \(\ell(\pi)\)
for the number of inversions of \(\pi\). The \emph{(right) weak Bruhat order}
\(\le_B\) on \(\Sym_n\) is the partial order generated by the covering relations
\[
\pi \lessdot_B \pi s_i,
\]
where \(s_i=(i,i+1)\) is the adjacent transposition and $\ell(\pi s_i)=\ell(\pi)+1$. Equivalently,
\[
\pi \lessdot_B \pi s_i
\quad\text{if and only if}\quad
\pi_i<\pi_{i+1},
\]
since right multiplication by \(s_i\) swaps the entries in positions \(i\) and
\(i+1\). For example, in Figure~\ref{fig:projection}, the weak Bruhat cover \(132\lessdot_B 312\) is collapsed by \(\phi\) to a single tree in \(\Y_3\). The \emph{(left) weak Bruhat} order is defined analogously using left multiplication: its covers replace \(\pi\) with \(s_i\pi\) whenever \(\ell(s_i\pi)=\ell(\pi)+1\). In one-line notation this swaps \(i\) and \(i+1\), and the length increases exactly when \(i\) appears to the left of \(i+1\). Thus Figure~\ref{fig:projection} also shows that the left weak order cover from \(132\) to \(231\) is collapsed by the Loday--Ronco map. In the remainder of the paper, we will use ``weak order'' to mean ``right weak order''.

The set \(\Y_n\) has cardinality equal to the \(n\)-th Catalan number $C_n=\frac{1}{n+1}\binom{2n}{n}$.  It is partially ordered by the \emph{Tamari order} \(\le_T\), generated by the local rotation
\[
((X,Y),Z)\longrightarrow (X,(Y,Z)),
\]
where \(X\), \(Y\), and \(Z\) are binary trees, and the rotation may occur at any subtree.

\subsection{The Loday--Ronco map}

%Loday and Ronco~\cite{LoR} introduced the surjective map
%\[
%\psi:\Sym_n\to\Y_n
%\]
%from permutations to binary trees.

\begin{definition}
We define \(\psi\) recursively. For the empty permutation \(\emptyset\), let $\psi(\emptyset)=\ast\in\Y_0$. Let \(\pi=\pi_1\cdots\pi_n\in\Sym_n\) with \(n>0\), and let \(k\) be the unique index such that \(\pi_k=n\). Define
    \[
    \pi_L=\pi_1\cdots\pi_{k-1},
    \qquad
    \pi_R=\pi_{k+1}\cdots\pi_n.
    \]
    Then
    \[
    \psi(\pi)=\psi(\std(\pi_L))\vee\psi(\std(\pi_R)).
    \]
%where, for \(T_1\in\Y_p\) and \(T_2\in\Y_q\), the tree \(T_1\vee T_2\in\Y_{p+q+1}\) is obtained by attaching \(T_1\) and \(T_2\) as the left and right subtrees of a new root, respectively.
\end{definition}

\begin{exm}
Let \(\pi=2413\). The maximum entry is \(4\), occurring in position \(2\). Thus, $\pi_L=2$ and $\pi_R=13$, so
\[
\psi(2413)=\psi(\std(2))\vee\psi(\std(13))
=\psi(1)\vee\psi(12).
\]
Now
\[
\psi(1)=\psi(\emptyset)\vee\psi(\emptyset),
\qquad
\psi(12)=\psi(1)\vee\psi(\emptyset).
\]
Hence \(\psi(2413)\) is the binary tree shown below, obtained by grafting the tree
\(\psi(1)\) to the left of the root and the tree \(\psi(12)\) to the right.

{\centering
\includegraphics[scale=1.4]{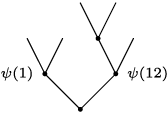}\par}
\end{exm}

\subsection{Tonks' Quotient Map}

%Tonks~\cite{Ton} defined a cellular quotient map $\theta:\Prm^n\longrightarrow \Asc^{n+1}$. 
Combinatorially, the faces of \(\Prm^n\) are indexed by ordered partitions of \([n]\), while the faces of \(\Asc^{n+1}\) are indexed by planar rooted trees
with \(n+1\) leaves. At the level of vertices, ordered partitions reduce to
permutations and planar rooted trees reduce to planar binary trees. Since only this
vertex-level restriction is used below, we pass to the induced map $\phi:\Sym_n\longrightarrow \Y_n$.

\begin{definition}\label{def:independence-perm}
Let \(\pi=\pi_1\cdots\pi_n\in\Sym_n\), and let \((\pi_i,\pi_{i+1})=(a,b)\) be a pair of entries in adjacent positions. We say that \((a,b)\) is \emph{Tonks-independent} if there exists an index \(j>i+1\) such that
\[
\min(a,b)<\pi_j<\max(a,b).
\]
Equivalently, \((a,b)\) is Tonks-independent if and only if some entry to the right of the pair has value strictly between \(a\) and \(b\).
\end{definition}

\begin{remark}
Tonks defines independence more generally for adjacent blocks \(A_{k-1},A_k\) in an ordered partition \((A_1,\dots,A_m)\). In that setting, the condition is that some element of a block strictly to the right of \(A_k\) lie strictly between the minimum and the maximum of \(A_{k-1}\cup A_k\). For vertices, where all blocks are singletons, this reduces to Definition~\ref{def:independence-perm}.
\end{remark}

Let \(\sim\) be the equivalence relation on \(\Sym_n\) generated by the
elementary moves that replace a permutation \(\pi\) by a permutation \(\rho\)
obtained from \(\pi\) by swapping a Tonks-independent pair. Tonks showed~\cite{Ton} that the
quotient \(\Sym_n/\!\sim\) can be canonically identified with \(\Y_n\). Under
this identification, \(\phi:\Sym_n\to\Y_n\) sends each permutation to the binary
tree corresponding to its \(\sim\)-equivalence class. For example, in
Figure~\ref{fig:projection}, the permutations \(132\) and \(312\) are
equivalent under \(\sim\), since the pair \((1,3)\) in \(132\) is
Tonks-independent.

\section{The Free Non-Symmetric Operad}

To describe Tonks' and Loday--Ronco maps syntactically, we introduce a free non-symmetric, non-unital operad on one binary generator (for general background on operads, see~\cite{LoV}).

\begin{definition}
Let \(\mathcal L\) be the term language generated by a single symbol \(\mathbf 2\) of arity \(2\), as follows:
\begin{itemize}
    \item [$\cdot$] \(\mathbf 2\) is a term of arity $|\mathbf{2}|=2$;
    \item [$\cdot$] if \(A\) and \(B\) are terms and \(1\le n\le |A|\), then \(A\circ_n B\) is a term of arity $|A\circ_n B|=|A|+|B|-1$.
\end{itemize}
%Here the arity \(|A|\) of a term \(A\) is defined recursively by
%\[
%|\mathbf 2|=2,
%\qquad
%|A\circ_n B|=|A|+|B|-1.
%\]
\end{definition}

Each term in \(\mathcal L\) determines a binary tree whose leaves are ordered from left to right. Under this interpretation, the generator \(\mathbf 2\) corresponds to the \(2\)-corolla, i.e., the unique binary tree with two leaves, and the term \(A\circ_n B\) corresponds to the tree obtained by grafting the root of the tree represented by \(B\) onto the \(n\)-th leaf of the tree represented by \(A\).

\subsection{The Equational Calculus \(\mathcal I\)}

We introduce the equational calculus \(\mathcal I\) on the term language \(\mathcal L\). It expresses the laws of partial composition in the free non-symmetric, non-unital operad on one binary generator. Our partial composition \(\circ_n\) corresponds to \(\triangleleft_n\) of Do\v{s}en--Petri\'c \cite{DP2} . The calculus \(\mathcal I\) used here is their unit-free calculus \(\mathcal I''\); in particular, we use only \emph{(assoc1)} and \emph{(assoc2)} and no unit axioms.

Let \(=_{\mathcal I}\) be the minimal congruence on \(\mathcal L\) satisfying the following two axioms:
\begin{align}
    \label{eq:assoc1} \tag{assoc1}
    (A \circ_n B) \circ_m C &= A \circ_n (B \circ_{m-n+1} C)
    && \text{if } n \le m < n+|B|, \\
    \label{eq:assoc2} \tag{assoc2}
    (A \circ_n B) \circ_m C &= (A \circ_{m-|B|+1} C) \circ_n B
    && \text{if } n+|B| \le m.
\end{align}

Axiom \emph{(assoc1)} corresponds to the case in which the grafting of $C$
lands
inside the subtree contributed by \(B\), while \emph{(assoc2)} corresponds to the
case in which the graftings of \(B\) and \(C\) occur at disjoint leaves of \(A\).
The following picture illustrates these two situations:

\begin{center}
\includegraphics{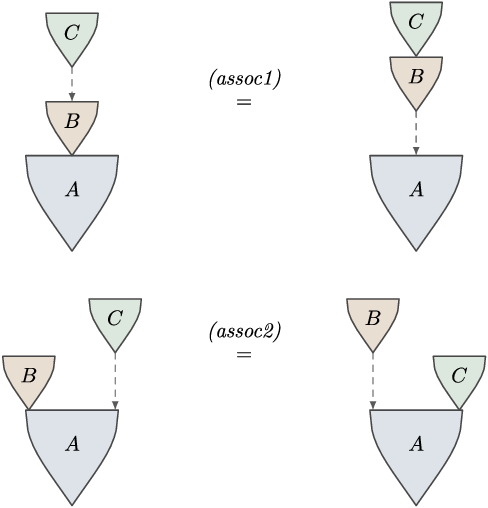}
\end{center}

A standard result from~\cite{DP1} says that two terms of \(\mathcal L\) are equal
modulo \(\mathcal I\) if and only if they determine the same binary tree.

\subsection{Indexed \(l\)-Factors}

To relate the term language \(\mathcal L\) to permutations, we pass to an indexed
version \(\mathcal L^I\), in which the generator \(\mathbf 2\) is replaced by
countably many indexed generators \(\mathbf 2^k\), for \(k\in\mathbb N^+\). We require that no index occurs
more than once in a term. For \(\mathbf A\in\mathcal L^I\), let
\(\operatorname{ind}(\mathbf A)\) denote the set of indices occurring in
\(\mathbf A\).

We now define the recursively generated class of \(l\)-factors used throughout
the paper.

\begin{definition}\label{def:l-factor}
A term \(\mathbf A\in\mathcal L^I\) is an \emph{\(l\)-factor} if it is obtained
recursively as follows:
\begin{itemize}
    \item [$\cdot$] \(\mathbf 2^k\) is an \(l\)-factor for every \(k\in\mathbb N^+\);
    \item [$\cdot$] if \(\mathbf A\) is an \(l\)-factor, \(j\notin\operatorname{ind}(\mathbf A)\), and
    \[
    r=\bigl|\{x\in\operatorname{ind}(\mathbf A):x<j\}\bigr|+1,
    \]
    then \(\mathbf A\circ_r \mathbf 2^j\) is an \(l\)-factor.
\end{itemize}
\end{definition}

The insertion position \(r\) is the rank of \(j\) in the increasing order on
\(\operatorname{ind}(\mathbf A)\cup\{j\}\). 
%Equivalently, it is one plus the number of indices already occurring in \(\mathbf A\) that are smaller than \(j\). 
In particular, \(1\le r\le |\mathbf A|\). For an \(l\)-factor
\(\mathbf A\), define its \emph{root index} recursively by
\[
\operatorname{root}(\mathbf 2^i)=i,
\qquad
\operatorname{root}(\mathbf A\circ_r \mathbf 2^j)=\operatorname{root}(\mathbf A).
\]

The following normalization lemma is central to what follows.

\begin{lemma}\label{lem:normalization}
Let \(\mathbf A\) be an \(l\)-factor, $i=\operatorname{root}(\mathbf A)$, and let
\[
L=\{x\in\operatorname{ind}(\mathbf A):x<i\}
\qquad
\text{and}
\qquad
R=\{x\in\operatorname{ind}(\mathbf A):x>i\}.
\]
Let \(\mathbf A_L\) and \(\mathbf A_R\) be $l$-factors with $\operatorname{ind}(\mathbf A_L)=L$ and $\operatorname{ind}(\mathbf A_R)=R$.
Then $\mathbf{A}$ is equivalent modulo ${\mathcal I}$ to a term in one of the following forms:
\begin{itemize}
    \item [$\cdot$] if $L,R\neq\emptyset$, then $\mathbf A =_{\mathcal I} (\mathbf 2^{i}\circ_2 \mathbf A_R)\circ_1 \mathbf A_L$;
    \item [$\cdot$] if $L=\emptyset$, then $\mathbf A =_{\mathcal I} \mathbf 2^{i}\circ_2 \mathbf A_R$;
    \item [$\cdot$] if $R=\emptyset$, then $\mathbf A =_{\mathcal I} \mathbf 2^{i}\circ_1 \mathbf A_L$;
    \item [$\cdot$] if $\operatorname{ind}(\mathbf A)=\{i\}$, then \(\mathbf A=\mathbf 2^{i}\).
\end{itemize}
\end{lemma}

\begin{proof}
We argue by induction on $N(\mathbf A):=\bigl|\operatorname{ind}(\mathbf A)\bigr|$. If \(N(\mathbf A)=1\), then \(\mathbf A=\mathbf 2^{i}\), and the claim is immediate. Assume now that \(N(\mathbf A)>1\). By Definition~\ref{def:l-factor}, we may write $\mathbf A=\mathbf B\circ_r \mathbf 2^x$, where \(\mathbf B\) is an \(l\)-factor, \(x\notin\operatorname{ind}(\mathbf B)\), and $r=\bigl|\{y\in\operatorname{ind}(\mathbf B):y<x\}\bigr|+1$. Moreover, $\operatorname{root}(\mathbf A)=\operatorname{root}(\mathbf B)=i$. Since \(i\in\operatorname{ind}(\mathbf B)\) and \(x\notin\operatorname{ind}(\mathbf B)\), we have $x\neq i$.

By the inductive hypothesis, one of the following holds for \(\mathbf B\):

\noindent (1) There exist \(l\)-factors \(\mathbf B_L,\mathbf B_R\) such that
\[
\operatorname{ind}(\mathbf B_L)=\{y\in\operatorname{ind}(\mathbf B):y<i\},\qquad
\operatorname{ind}(\mathbf B_R)=\{y\in\operatorname{ind}(\mathbf B):y>i\},\qquad
\mathbf B =_{\mathcal I} (\mathbf 2^{i}\circ_2 \mathbf B_R)\circ_1 \mathbf B_L.
\]

\noindent (2) There exists an \(l\)-factor \(\mathbf B_R\) such that
\[
\operatorname{ind}(\mathbf B_R)=\{y\in\operatorname{ind}(\mathbf B):y>i\},\qquad
\mathbf B =_{\mathcal I} \mathbf 2^{i}\circ_2 \mathbf B_R.
\]

\noindent (3) There exists an \(l\)-factor \(\mathbf B_L\) such that
\[
\operatorname{ind}(\mathbf B_L)=\{y\in\operatorname{ind}(\mathbf B):y<i\},\qquad
\mathbf B =_{\mathcal I} \mathbf 2^{i}\circ_1 \mathbf B_L.
\]

\noindent (4) \(\mathbf B=\mathbf 2^{i}\).

We distinguish two cases.

\medskip
\noindent\emph{Case 1: \(x<i\).}
Then every index greater than \(i\) is also greater than \(x\). If
\[
\mathbf B =_{\mathcal I} (\mathbf 2^{i}\circ_2 \mathbf B_R)\circ_1 \mathbf B_L,
\]
then
\[
\{y\in\operatorname{ind}(\mathbf B):y<x\}
=
\{y\in\operatorname{ind}(\mathbf B_L):y<x\},
\]
so the insertion index \(r\) is exactly the one prescribed for inserting \(x\) into \(\mathbf B_L\). In particular, $1\le r\le |\mathbf B_L|$.
Applying \emph{(assoc1)} gives
\[
((\mathbf 2^{i}\circ_2 \mathbf B_R)\circ_1 \mathbf B_L)\circ_r \mathbf 2^x
=
(\mathbf 2^{i}\circ_2 \mathbf B_R)\circ_1(\mathbf B_L\circ_r \mathbf 2^x).
\]
Set
\[
\mathbf A_L:=\mathbf B_L\circ_r \mathbf 2^x,
\qquad
\mathbf A_R:=\mathbf B_R.
\]
Then \(\mathbf A_L\) is an \(l\)-factor, \(\mathbf A_R\) is an \(l\)-factor, and
\[
\operatorname{ind}(\mathbf A_L)=\{y\in\operatorname{ind}(\mathbf A):y<i\},
\qquad
\operatorname{ind}(\mathbf A_R)=\{y\in\operatorname{ind}(\mathbf A):y>i\}.
\]
Thus
\[
\mathbf A =_{\mathcal I} (\mathbf 2^{i}\circ_2 \mathbf A_R)\circ_1 \mathbf A_L.
\]
If either 
\[
\mathbf B =_{\mathcal I} \mathbf 2^{i}\circ_2 \mathbf B_R \qquad \text{or} \qquad \mathbf B =_{\mathcal I} \mathbf 2^{i}\circ_1 \mathbf B_L,
\]
the cases are proven similarly, while the case where \(\mathbf B=\mathbf 2^{i}\) is trivial.

\medskip
\noindent\emph{Case 2: \(x>i\).} Then every index smaller than \(i\) is also smaller than \(x\). If
\[
\mathbf B =_{\mathcal I} (\mathbf 2^{i}\circ_2 \mathbf B_R)\circ_1 \mathbf B_L,
\]
let
\[
s=\bigl|\{y\in\operatorname{ind}(\mathbf B_R):y<x\}\bigr|+1.
\]
Since every index in \(\mathbf B_L\), as well as the root index \(i\), is smaller than \(x\), we have $r=|\mathbf B_L|+s$ and $1+|\mathbf B_L|\le r$. Hence \emph{(assoc2)} gives
\[
((\mathbf 2^{i}\circ_2 \mathbf B_R)\circ_1 \mathbf B_L)\circ_r \mathbf 2^x
=
\bigl((\mathbf 2^{i}\circ_2 \mathbf B_R)\circ_{s+1}\mathbf 2^x\bigr)\circ_1 \mathbf B_L.
\]
Now \(2\le s+1<2+|\mathbf B_R|\), so \emph{(assoc1)} yields
\[
(\mathbf 2^{i}\circ_2 \mathbf B_R)\circ_{s+1}\mathbf 2^x
=
\mathbf 2^{i}\circ_2(\mathbf B_R\circ_s \mathbf 2^x).
\]
Set
\[
\mathbf A_L:=\mathbf B_L,
\qquad
\mathbf A_R:=\mathbf B_R\circ_s \mathbf 2^x.
\]
Then \(\mathbf A_L\) and \(\mathbf A_R\) are \(l\)-factors, and
\[
\operatorname{ind}(\mathbf A_L)=\{y\in\operatorname{ind}(\mathbf A):y<i\},
\qquad
\operatorname{ind}(\mathbf A_R)=\{y\in\operatorname{ind}(\mathbf A):y>i\}.
\]
Thus
\[
\mathbf A =_{\mathcal I} (\mathbf 2^{i}\circ_2 \mathbf A_R)\circ_1 \mathbf A_L.
\]
Again, if either
\[
\mathbf B =_{\mathcal I} \mathbf 2^{i}\circ_2 \mathbf B_R \qquad \text{or} \qquad \mathbf B =_{\mathcal I} \mathbf 2^{i}\circ_1 \mathbf B_L,
\]
the proofs are analogous, with the last case \(\mathbf B=\mathbf 2^{i}\) being trivial.

\end{proof}

\subsection{Soundness of the axioms under evaluation}\label{subsec:soundness-eval}

Let
\[
\Y=\bigcup_{n\ge 0}\Y_n.
\]
We define an evaluation map
\[
\eval:\mathcal L^I\to\Y
\]
by forgetting the indices on the generators and interpreting each \(\mathbf 2^k\) as the $2-$corolla, and each partial composition \(\circ_i\) as grafting at the \(i\)-th leaf. Thus \(\eval(\mathbf A)\) is the binary tree determined by the term \(\mathbf A\). We use the same symbol \(=_{\mathcal I}\) for the congruence on \(\mathcal L^I\)
generated by the same axiom schemes \emph{(assoc1)} and \emph{(assoc2)}.

\begin{proposition}\label{prop:eval-congruence}
If \(t =_{\mathcal I} t'\), then \(\eval(t)=\eval(t')\). Equivalently, each axiom scheme \emph{(assoc1)} and \emph{(assoc2)} preserves the evaluated binary tree.
\end{proposition}

\begin{proof}
By definition of \(\eval\), the term
\((\mathbf A\circ_n \mathbf B)\circ_m \mathbf C\) is evaluated by first
grafting the tree \(\eval(\mathbf B)\) onto the \(n\)-th leaf of
\(\eval(\mathbf A)\), and then grafting \(\eval(\mathbf C)\) onto the
appropriate leaf of the resulting tree. In axiom \emph{(assoc1)}, the second grafting lands inside the subtree
contributed by \(\mathbf B\). Thus the two sides describe the same iterated
grafting and evaluate to the same binary tree. In axiom \emph{(assoc2)}, the graftings of \(\mathbf B\) and \(\mathbf C\)
occur at disjoint leaves of \(\eval(\mathbf A)\). Therefore the order in which
these two graftings are performed does not affect the resulting binary tree,
and the two sides again have the same evaluation.

Since \(=_{\mathcal I}\) is the congruence generated by these axioms, and since
\(\eval\) is compatible with partial composition, \(\eval\) is constant on
\(\mathcal I\)-equivalence classes.
\end{proof}

Do\v{s}en and Petri\'c~\cite{DP1} show for the unit-free calculus \(\mathcal I''\) that two terms are equal in the calculus \emph{if and only if} they determine the same binary tree. In the present paper, only the soundness direction of Proposition~\ref{prop:eval-congruence} is needed.

\section{The Maps as Syntactic Projections}

We now identify Tonks' map and the Loday--Ronco map with the two syntactic
encodings defined below.

\subsection{Encoding Permutations}

We define two encodings of permutations in the indexed language \(\mathcal L^I\). The first is the head-insertion encoding, and the second is the decreasing encoding.

\begin{definition}\label{def:head-insertion}
Let \(a=a_1\cdots a_n\) be a word of distinct positive integers. We define
\(h(a)\in\mathcal L^I\) recursively. If \(n=1\), then
    \[
    h(a_1)=\mathbf 2^{a_1}.
    \]
If $n>1$, let \(a'=a_1\cdots a_{n-1}\). Then
    \[
    h(a)=h(a')\circ_r \mathbf 2^{a_n},
    \]
    where
    \[
    r=1+\bigl|\{\,a_j \text{ occurring in } a' : a_j<a_n\,\}\bigr|.
    \]
\end{definition} 
In particular, if \(\pi\in\Sym_n\), then \(h(\pi)\) is obtained by applying this
construction to \(\pi\) viewed as a word on $[n]$. The recursive definition of \(h\) immediately gives the following fact.

\begin{lemma}\label{lem:h-is-lfactor}
For every word $a=a_1\cdots a_n$ of distinct positive integers, the term \(h(a)\) is an \(l\)-factor. Moreover,
\[
\operatorname{ind}(h(a))=\{a_1,\dots,a_n\},
\qquad
\operatorname{root}(h(a))=a_1.
\]
\end{lemma}

We now describe the second encoding.
\begin{definition}\label{def:decreasing-encoding}
Let \(a=a_1\cdots a_r\) be a nonempty word of distinct positive integers, and
let
\[
\kappa(a)=\kappa_1\cdots\kappa_r
\]
be the decreasing rearrangement of the letters of \(a\), so that
\[
\kappa_1>\kappa_2>\cdots>\kappa_r.
\]
For $x$\ occurring in $a$, let
\[
u_a(x)=
\bigl|\{\,a_j \text{ to the left of } x \text{ in } a : a_j>x\,\}\bigr|.
\]
Define \(f(a)\in\mathcal L^I\) as follows. If \(r=1\), set
\[
f(a)=\mathbf 2^{\kappa_1}.
\]
If \(r>1\), set
\[
f(a)
=
(\cdots((\mathbf 2^{\kappa_1}\circ_{m_{\kappa_2}}\mathbf 2^{\kappa_2})
\circ_{m_{\kappa_3}}\mathbf 2^{\kappa_3})\cdots)
\circ_{m_{\kappa_r}}\mathbf 2^{\kappa_r},
\]
where, for \(q=2,\dots,r\),
\[
m_{\kappa_q}=u_a(\kappa_q)+1.
\]
This is well defined: when \(\mathbf 2^{\kappa_q}\) is inserted, the term
already contains the generators whose indices are entries of \(a\)
greater than \(\kappa_q\), namely \(\kappa_1,\dots,\kappa_{q-1}\). Hence the
current arity is \(q\), while
\[
0\le u_a(\kappa_q)\le q-1.
\]
Therefore $1\le m_{\kappa_q}\le q$, so the insertion index is valid.
\end{definition}

In particular, if \(\pi\in\Sym_n\), then \(\kappa(\pi)=n(n-1)\cdots1\), and
the preceding definition gives
\[
f(\pi)
=
(\cdots((\mathbf 2^n\circ_{m_{n-1}}\mathbf 2^{n-1})
\circ_{m_{n-2}}\mathbf 2^{n-2})\cdots)\circ_{m_1}\mathbf 2^1,
\]
where
\[
m_x=u_\pi(x)+1
\qquad (1\le x<n).
\]

\begin{exm}
We illustrate the two encodings on the permutation \(\pi=312\in\Sym_3\). For the head-insertion encoding \(h\), we first compute
\[
%h(31)=\mathbf 2^3\circ_1 \mathbf 2^1,
h(312)=h(31)\circ_2\mathbf 2^2=(\mathbf 2^3\circ_1 \mathbf 2^1)\circ_2 \mathbf 2^2.
\]
%since among the letters of the word \(3\), none is smaller than \(1\), so the
%insertion index is $1+0=1$. Then among the letters of the word \(31\), exactly one is smaller than \(2\), so
%the insertion index is $1+1=2$. Hence
%\[
%h(312)=(\mathbf 2^3\circ_1 \mathbf 2^1)\circ_2 \mathbf 2^2.
%\]

To compute the decreasing encoding \(f\), we process the values in the order \(3,2,1\). Since $u_\pi(2)=1$ and $u_\pi(1)=1$ we have
\[
m_2=u_\pi(2)+1=2,
\qquad
m_1=u_\pi(1)+1=2.
\]
Hence
\[
f(312)=(\mathbf 2^3\circ_2 \mathbf 2^2)\circ_2 \mathbf 2^1.
\]
\end{exm}

\subsection{Tonks' Projection via \(h\)}

We now define a recursive insertion map on permutations and show that it agrees
with Tonks' vertex map \(\phi\). Let \(t\in\Y_1\) be the $2-$corolla. For
\(T\in\Y_p\) and \(1\le i\le p+1\), let
\[
T\triangleleft_i t \in \Y_{p+1}
\]
denote the binary tree obtained by grafting the $2-$corolla \(t\) at
the \(i\)-th leaf of \(T\).

\begin{definition}\label{def:phihat}
For each \(n\ge 0\), define $\widehat\phi:\Sym_n\to\Y_n$ recursively. Set $\widehat\phi(\emptyset)=\ast$, where \(\ast\in\Y_0\) is the unique tree with a single vertex. $\widehat\phi(\pi)=t$, if $\pi\in\Sym_1$, and
\[
\widehat\phi(\pi)=\widehat\phi(\std(\pi_2\cdots \pi_n))\triangleleft_{\pi_1} t
\]
for \(\pi=\pi_1\cdots \pi_n\in\Sym_n\) with \(n>1\).
\end{definition}

This is well defined, since \(\std(\pi_2\cdots \pi_n)\in\Sym_{n-1}\), the tree
\(\widehat\phi(\std(\pi_2\cdots \pi_n))\) lies in \(\Y_{n-1}\), and therefore
has exactly \(n\) leaves, while \(\pi_1\in[n]\).

At the level of vertices, Tonks' construction may be described (as in \cite{Ton}) by the following
evaluation procedure for a product \(x_1x_2\cdots x_{n+1}\). An ordered partition of \([n]\) records how the \(n\) binary compositions are
grouped into stages. On vertices, where the partition consists of singletons,
this procedure is determined by a permutation  $\pi=\pi_1\cdots\pi_n\in\Sym_n$, and the compositions are carried out successively in the order
\(\pi_1,\dots,\pi_n\). In particular, the first step composes the variables
\(x_{\pi_1}\) and \(x_{\pi_1+1}\). On the tree side, this corresponds to
grafting a $2-$corolla at the \(\pi_1\)-st leaf. After this first
composition, the remaining \(n-1\) steps act on a product of length \(n\), so
their relative order is encoded by the standardized tail $\std(\pi_2\cdots\pi_n)$. We record this in the following proposition.

\begin{proposition}\label{prop:phihat-equals-phi}
For every permutation \(\pi\in\Sym_n\), $\widehat\phi(\pi)=\phi(\pi)$.
\end{proposition}

Although \(\widehat\phi\) is introduced only as an auxiliary map, it plays an
important role in what follows. It leads to an alternative recursive
description of Tonks' map, which we denote by \(\varphi\). This formulation is
better suited to the operadic arguments developed later, and in particular it
will allow us to use the normalization lemma for \(l\)-factors proved earlier
in the paper in the proof of our main result on Tonks' map.

\begin{definition}\label{def:varphi}
For each \(n\ge 0\), define \(\varphi:\Sym_n\to\Y_n\) recursively by
\[
\varphi(\emptyset)=\ast,
\]
where \(\ast\in\Y_0\) is the unique tree with a single vertex. For
\[
\pi=\pi_1\cdots\pi_n\in\Sym_n
\qquad (n>0),
\]
set
\[
\varphi(\pi)=
\varphi(\std(\pi^{<\pi_n}))\vee
\varphi(\std(\pi^{>\pi_n})),
\]
where \(\vee\) denotes grafting at a new root. If one of the two subsequences is
empty, the corresponding tree is \(\ast\in\Y_0\).
\end{definition}

\begin{theorem}\label{thm:phihat-equals-varphi}
For every permutation \(\pi\in\Sym_n\), $\widehat\phi(\pi)=\varphi(\pi)$.
\end{theorem}
\begin{proof}
We argue by induction on \(n\). If \(n=0\), then both \(\widehat\phi(\emptyset)\) and \(\varphi(\emptyset)\) are
equal to the unique tree \(\ast\in\Y_0\). If \(n=1\), then both maps send the
unique permutation in \(\Sym_1\) to the unique tree \(t\in\Y_1\). Assume now that \(n>1\), and let
\[
\pi=\pi_1\cdots \pi_n\in\Sym_n.
\]
Set
\[
\rho=\std(\pi_2\cdots \pi_n)\in\Sym_{n-1}.
\]
By Definition~\ref{def:phihat},
\[
\widehat\phi(\pi)=\widehat\phi(\rho)\triangleleft_{\pi_1} t.
\]
By the inductive hypothesis,
\[
\widehat\phi(\rho)=\varphi(\rho).
\]
Thus
\[
\widehat\phi(\pi)=\varphi(\rho)\triangleleft_{\pi_1} t.
\]
Since \(\rho\in\Sym_{n-1}\), the recursive definition of \(\varphi\) yields
\[
\varphi(\rho)
=
\varphi(\std(\rho^{<\rho_{n-1}}))\vee
\varphi(\std(\rho^{>\rho_{n-1}})).
\]
We distinguish two cases.

\medskip
\noindent\emph{Case 1: \(\pi_1<\pi_n\).}
Then \(\rho_{n-1}=\pi_n-1\). Let
\[
L=\varphi(\std(\rho^{<\rho_{n-1}})),
\qquad
R=\varphi(\std(\rho^{>\rho_{n-1}})).
\]
Thus
\[
\varphi(\rho)=L\vee R.
\]
The entries of \(\rho\) smaller than \(\rho_{n-1}\) are exactly the standardized
images of the entries of \(\pi\) that are smaller than \(\pi_n\), except for the
first entry \(\pi_1\) itself. Hence, if
\[
\alpha=\std(\pi^{<\pi_n}),
\]
then \(\alpha_1=\pi_1\) and
\[
\std(\alpha_2\cdots\alpha_{\pi_n-1})=\std(\rho^{<\rho_{n-1}}).
\]
By Definition~\ref{def:phihat} and the inductive hypothesis,
\[
L\triangleleft_{\pi_1} t
=
\widehat\phi(\alpha)
=
\varphi(\alpha).
\]
Moreover, the entries of \(\rho\) greater than \(\rho_{n-1}\) are exactly the
standardized images of the entries of \(\pi\) greater than \(\pi_n\), so
\[
R=\varphi(\std(\pi^{>\pi_n})).
\]
Since \(L\in\Y_{\pi_n-2}\), the tree \(L\) has exactly \(\pi_n-1\) leaves.
As \(\pi_1<\pi_n\), the insertion at the global leaf \(\pi_1\) of \(L\vee R\)
therefore occurs inside the left subtree. Therefore
\[
(L\vee R)\triangleleft_{\pi_1} t
=
(L\triangleleft_{\pi_1} t)\vee R.
\]
Hence
\[
\widehat\phi(\pi)
=
\varphi(\std(\pi^{<\pi_n}))\vee \varphi(\std(\pi^{>\pi_n}))
=
\varphi(\pi).
\]

\medskip
\noindent\emph{Case 2: \(\pi_1>\pi_n\).}
Then \(\rho_{n-1}=\pi_n\). Let
\[
L=\varphi(\std(\rho^{<\rho_{n-1}})),
\qquad
R=\varphi(\std(\rho^{>\rho_{n-1}})).
\]
In this case
\[
L=\varphi(\std(\pi^{<\pi_n})).
\]
Let
\[
\beta=\std(\pi^{>\pi_n}).
\]
Then the first entry of \(\beta\) is \(\pi_1-\pi_n\), and the standardized tail
of \(\beta\) is \(\std(\rho^{>\rho_{n-1}})\). Hence, by
Definition~\ref{def:phihat} and the inductive hypothesis,
\[
R\triangleleft_{\pi_1-\pi_n}t
=
\widehat\phi(\beta)
=
\varphi(\beta).
\]
Since \(L\) has exactly \(\pi_n\) leaves, insertion at the global leaf
\(\pi_1\) of \(L\vee R\) occurs in the right subtree, at leaf
\(\pi_1-\pi_n\). Therefore
\[
(L\vee R)\triangleleft_{\pi_1}t
=
L\vee(R\triangleleft_{\pi_1-\pi_n}t).
\]
It follows that
\[
\widehat\phi(\pi)
=
\varphi(\std(\pi^{<\pi_n}))\vee
\varphi(\std(\pi^{>\pi_n}))
=
\varphi(\pi).
\]
\end{proof}

We now specialize Lemma~\ref{lem:normalization} to the head-insertion
encoding. The general lemma gives a root decomposition for every \(l\)-factor.
For terms of the form \(h(a)\), the additional point is that the indices of left and right terms in the decomposition
inherit the order of the corresponding subwords of \(a\).

\begin{lemma}\label{lem:h-word-decomposition}
Let \(a=a_1\cdots a_n\) be a nonempty word of distinct positive integers, let
\(i=a_1\), and let \(a^{<i}\) and \(a^{>i}\) denote the subwords of \(a\)
consisting of the entries less than \(i\) and greater than \(i\), respectively.
Then $h(a)$ is equivalent modulo ${\mathcal I}$ to a term in one of the following forms:
\begin{itemize}
    \item [$\cdot$] if \(a^{<i}, a^{>i}\neq\emptyset\), then $h(a)=_{\mathcal I}(\mathbf 2^i\circ_2 h(a^{>i}))\circ_1 h(a^{<i})$;
    \item [$\cdot$] if \(a^{<i}=\emptyset\), then $h(a)=_{\mathcal I}\mathbf 2^i\circ_2 h(a^{>i})$;
    \item [$\cdot$] if \(a^{>i}=\emptyset\), then $h(a)=_{\mathcal I}\mathbf 2^i\circ_1 h(a^{<i})$;
    \item [$\cdot$] if \(a=i\), then $h(a)=\mathbf 2^i$.
\end{itemize}
\end{lemma}

\begin{proof}
By Lemma~\ref{lem:h-is-lfactor}, \(h(a)\) is an \(l\)-factor and $i=\operatorname{root}(h(a))=a_1$. Lemma~\ref{lem:normalization} gives a root decomposition of \(h(a)\)
\[
h(a)=_{\mathcal I}(\mathbf 2^i\circ_2 \mathbf A_R)\circ_1 \mathbf A_L,
\]
with the evident one-sided variants, where
\[
\operatorname{ind}(\mathbf A_L)=\{x\in\operatorname{ind}(h(a)):x<i\},
\qquad
\operatorname{ind}(\mathbf A_R)=\{x\in\operatorname{ind}(h(a)):x>i\}.
\]
It remains only to check that, for the particular term \(h(a)\), the indices
appearing in \(\mathbf A_L\) and \(\mathbf A_R\) occur in the order inherited from the word
\(a\). We argue by induction on the length of prefixes of \(a\). For \(1\leq t\leq n\), recall that we have defined $a_{\leq t}=a_1\cdots a_t$.

We claim that \(h(a_{\leq t})\) satisfies the decomposition stated in the lemma,
with \(a_{\leq t}^{<i}\) and \(a_{\leq t}^{>i}\) in place of \(a^{<i}\) and
\(a^{>i}\), respectively. The case \(t=1\) is immediate. Assume the claim for \(a_{\leq t}\), and write $a_{\leq t+1}=a_{\leq t}x$.

If \(x<i\), then
\[
a_{\leq t+1}^{<i}=a_{\leq t}^{<i}x,
\qquad
a_{\leq t+1}^{>i}=a_{\leq t}^{>i}.
\]
The insertion position of \(x\) in \(h(a_{\leq t+1})\) is $1+\bigl|\{y\in a_{\leq t}:y<x\}\bigr|$. Since every entry of \(a_{\leq t}^{>i}\) is greater than \(x\), this is exactly
the insertion position of \(x\) in \(h(a_{\leq t}^{<i})\). Hence the right-hand
term in the root decomposition is unchanged, while the left-hand term becomes
\[
h(a_{\leq t}^{<i}x)=h(a_{\leq t+1}^{<i}).
\]

The case \(x>i\) is analogous: the left-hand term is unchanged, and the insertion
position of \(x\) in the right-hand term is precisely the position prescribed in
the construction of
\[
h(a_{\leq t}^{>i}x)=h(a_{\leq t+1}^{>i}).
\]

Taking \(t=n\), the induction gives
\[
h(a)=_{\mathcal I}
(\mathbf 2^i\circ_2 h(a^{>i}))\circ_1 h(a^{<i})
\]
when both subwords are nonempty, and the corresponding one-sided
congruences when one of them is empty. These are precisely the four cases stated
in the lemma.
\end{proof}

We now show that the grafting map \(\varphi\) is realized by the syntactic
encoding \(h\) applied to the reversed permutation. For the rest of the paper, for every nonempty word \(a\) of distinct positive integers, write \(g(a)=\eval(h(w(a)))\).

\begin{theorem}\label{thm:varphi-via-h}
For every \(n\ge 1\) and every permutation \(\pi\in\Sym_n\), $\varphi(\pi)=g(\pi)$.
\end{theorem}

\begin{proof}
We argue by induction on \(n\). If \(n=1\), then \(\pi=1\), \(w(\pi)=1\), and
\(h(w(\pi))=\mathbf 2^1\). Hence \(g(\pi)\) is the $2-$corolla, which is
exactly \(\varphi(1)\).

Assume now that the statement holds for all permutations of size \(<n\), and let $\pi=\pi_1\cdots\pi_n\in\Sym_n$. Set
\[
a=w(\pi)=\pi_n\pi_{n-1}\cdots\pi_1.
\]
The first letter of \(a\) is \(\pi_n\). Moreover, the subword of \(a\) consisting
of letters smaller than \(\pi_n\) is \(w(\pi^{<\pi_n})\), and the subword
consisting of letters greater than \(\pi_n\) is \(w(\pi^{>\pi_n})\).

Suppose first that both \(\pi^{<\pi_n}\) and \(\pi^{>\pi_n}\) are nonempty. By
Lemma~\ref{lem:h-word-decomposition},
\[
h(w(\pi))
=_{\mathcal I}
\bigl(\mathbf 2^{\pi_n}\circ_2 h(w(\pi^{>\pi_n}))\bigr)
\circ_1 h(w(\pi^{<\pi_n})).
\]
Applying \(\eval\) and using Proposition~\ref{prop:eval-congruence}, we obtain
\[
g(\pi)
=
g(\pi^{<\pi_n})\vee g(\pi^{>\pi_n}).
\]
Standardization preserves the relative order of a word. Since the insertion
positions in \(h\) depend only on this relative order, \(h(w(a))\) and
\(h(w(\std(a)))\) have the same term shape and differ only in their
indices; after applying \(\eval\), these indices are forgotten.
Therefore
\[
g(\pi^{<\pi_n})=g(\std(\pi^{<\pi_n})),
\qquad
g(\pi^{>\pi_n})=g(\std(\pi^{>\pi_n})).
\]
By the inductive hypothesis,
\[
g(\std(\pi^{<\pi_n}))
=
\varphi(\std(\pi^{<\pi_n})),
\qquad
g(\std(\pi^{>\pi_n}))
=
\varphi(\std(\pi^{>\pi_n})).
\]
Hence
\[
g(\pi)
=
\varphi(\std(\pi^{<\pi_n}))\vee
\varphi(\std(\pi^{>\pi_n}))
=
\varphi(\pi).
\]

If \(\pi^{<\pi_n}=\emptyset\), then Lemma~\ref{lem:h-word-decomposition} gives
\[
h(w(\pi))
=_{\mathcal I}
\mathbf 2^{\pi_n}\circ_2 h(w(\pi^{>\pi_n})).
\]
After applying \(\eval\), this says
\[
g(\pi)=\ast\vee g(\pi^{>\pi_n}).
\]
By standardization and the inductive hypothesis,
\[
g(\pi^{>\pi_n})
=
g(\std(\pi^{>\pi_n}))
=
\varphi(\std(\pi^{>\pi_n})).
\]
Thus
\[
g(\pi)
=
\ast\vee\varphi(\std(\pi^{>\pi_n}))
=
\varphi(\pi).
\]

The case \(\pi^{>\pi_n}=\emptyset\) is analogous. Then
\[
h(w(\pi))
=_{\mathcal I}
\mathbf 2^{\pi_n}\circ_1 h(w(\pi^{<\pi_n})),
\]
and hence
\[
g(\pi)
=
\varphi(\std(\pi^{<\pi_n}))\vee\ast
=
\varphi(\pi).
\]

This completes the induction.
\end{proof}
\begin{corollary}\label{cor:phi-via-h}
For every \(n\ge 1\) and every permutation \(\pi\in\Sym_n\), $\phi(\pi)=g(\pi)$.
\end{corollary}

\begin{proof}
By Proposition~\ref{prop:phihat-equals-phi} and
Theorem~\ref{thm:phihat-equals-varphi}, we have $\phi(\pi)=\varphi(\pi)$.
The conclusion now follows from Theorem~\ref{thm:varphi-via-h}.
\end{proof}

\subsection{The Loday--Ronco Map via \(f\)}

We now relate the Loday--Ronco map \(\psi\) to the decreasing encoding \(f\). Recall that \(\psi\) is defined recursively by splitting a permutation \(\pi\) at its maximal entry \(n\). Thus, if viewed as a word on $[n]$ we have that $\pi=\pi_L\,n\,\pi_R$, and then
\[
\psi(\pi)=\psi(\std(\pi_L))\vee\psi(\std(\pi_R)).
\]
To identify \(\psi\) with the evaluation of \(f\), we first prove a normalization lemma for the decreasing encoding, analogous to Lemma~\ref{lem:normalization}.

\begin{lemma}\label{lem:f-word-normalization}
Let \(a=a_1\cdots a_r\) be a nonempty word of distinct positive integers, and
let \(\kappa(a)=\kappa_1\cdots\kappa_r\) be the decreasing rearrangement of the
letters of \(a\) and write
\[
a=a_L\,\kappa_1\,a_R,
\]
where \(a_L\) and \(a_R\) are the subwords to the left and to the right of
\(\kappa_1\), respectively. 
Let $\mathbf B_L =_{\mathcal I} f(a_L)$ and $\mathbf B_R =_{\mathcal I} f(a_R)$. Then, $f(a)$ is equivalent modulo ${\mathcal I}$ to a term in one of the following forms:
\begin{itemize}
    \item [$\cdot$] if $a_L,a_R\neq\emptyset$, then $f(a)=_{\mathcal I}(\mathbf 2^{\kappa_1}\circ_2 \mathbf B_R)\circ_1 \mathbf B_L$;
    \item [$\cdot$] if \(a_L=\emptyset\), then $f(a)=_{\mathcal I}\mathbf 2^{\kappa_1}\circ_2 \mathbf B_R$;
    \item [$\cdot$] if \(a_R=\emptyset\), then $f(a)=_{\mathcal I}\mathbf 2^{\kappa_1}\circ_1 \mathbf B_L$;
    \item [$\cdot$] if \(a=\kappa_1\), then \(f(a)=\mathbf 2^{\kappa_1}\).
\end{itemize}
\end{lemma}

\begin{proof}
We argue by induction on \(r=\operatorname{len}(a)\). If \(r=1\), then \(a=\kappa_1\) and $f(a)=\mathbf 2^{\kappa_1}$, so the claim is immediate. Assume now that \(r>1\). Let \(a'\) be the word
obtained from \(a\) by deleting its smallest letter \(\kappa_r\). Then
\[
f(a)=f(a')\circ_{u_a(\kappa_r)+1}\mathbf 2^{\kappa_r}.
\]
The maximum of \(a'\) is still \(\kappa_1\). Write
\[
a'=a'_L\,\kappa_1\,a'_R.
\]
Since each axiom scheme preserves arity, congruent terms have the same arity.

We distinguish two cases.

\medskip
\noindent\emph{Case 1: \(\kappa_r\text{ occurs in } a_L\).} Then \(a'_R=a_R\), while \(a'_L\) is obtained from \(a_L\) by deleting
\(\kappa_r\). Moreover,
\[
u_a(\kappa_r)=u_{a_L}(\kappa_r),
\]
because every letter of \(a\) that is greater than \(\kappa_r\) and lies to its
left is already contained in \(a_L\).

\smallskip
\noindent If both \(a'_L\) and \(a'_R\) are nonempty, then by the inductive
hypothesis there exist terms \(\mathbf A_L,\mathbf A_R\in\mathcal L^I\) such that
\[
\mathbf A_L =_{\mathcal I} f(a'_L),\qquad \mathbf A_R =_{\mathcal I} f(a_R),
\]
and
\[
f(a') =_{\mathcal I} (\mathbf 2^{\kappa_1}\circ_2 \mathbf A_R)\circ_1 \mathbf A_L.
\]
Since \(\mathbf A_L =_{\mathcal I} f(a'_L)\), we have
\[
|\mathbf A_L|=|f(a'_L)|=\operatorname{len}(a'_L)+1=\operatorname{len}(a_L).
\]
Hence
\[
1\le u_{a_L}(\kappa_r)+1\le |\mathbf A_L|,
\]
so \emph{(assoc1)} applies:
\[
((\mathbf 2^{\kappa_1}\circ_2 \mathbf A_R)\circ_1 \mathbf A_L)\circ_{u_a(\kappa_r)+1}\mathbf 2^{\kappa_r}
=
(\mathbf 2^{\kappa_1}\circ_2 \mathbf A_R)\circ_1
\bigl(\mathbf A_L\circ_{u_{a_L}(\kappa_r)+1}\mathbf 2^{\kappa_r}\bigr).
\]
Set
\[
\mathbf B_L:=\mathbf A_L\circ_{u_{a_L}(\kappa_r)+1}\mathbf 2^{\kappa_r},
\qquad
\mathbf B_R:=\mathbf A_R.
\]
Since \(\mathbf A_L =_{\mathcal I} f(a'_L)\), it follows from the definition of \(f(a_L)\)
that
\[
\mathbf B_L =_{\mathcal I} f(a_L).
\]
Therefore
\[
f(a)=_{\mathcal I}(\mathbf 2^{\kappa_1}\circ_2 \mathbf B_R)\circ_1 \mathbf B_L.
\]

\smallskip
\noindent If either \(a'_L=\emptyset\) and \(a'_R\neq\emptyset\), or \(a'_R=\emptyset\) and \(a'_L\neq\emptyset\), we proceed analogously, and if \(a'=\kappa_1\) this case is trivial.

\medskip
\noindent\emph{Case 2: \(\kappa_r\text{ occurs in } a_R\).} Then \(a'_L=a_L\), while \(a'_R\) is obtained from \(a_R\) by deleting
\(\kappa_r\). Put
\[
s=u_{a_R}(\kappa_r)+1.
\]
Since every letter of \(a_L\), as well as \(\kappa_1\) itself, lies to the left
of \(\kappa_r\) and is greater than \(\kappa_r\), we have
\[
u_a(\kappa_r)=\operatorname{len}(a_L)+u_{a_R}(\kappa_r)+1.
\]
Again, we prove only the main case, since the other cases are similar.
\smallskip
\noindent If both \(a'_L\) and \(a'_R\) are nonempty, then by the inductive
hypothesis there exist terms \(\mathbf A_L,\mathbf A_R\in\mathcal L^I\) such that
\[
\mathbf A_L =_{\mathcal I} f(a_L),\qquad \mathbf A_R =_{\mathcal I} f(a'_R),
\]
and
\[
f(a') =_{\mathcal I} (\mathbf 2^{\kappa_1}\circ_2 \mathbf A_R)\circ_1 \mathbf A_L.
\]
Since \(\mathbf A_L =_{\mathcal I} f(a_L)\), we have
\[
|\mathbf A_L|=|f(a_L)|=\operatorname{len}(a_L)+1.
\]
Therefore
\[
u_a(\kappa_r)+1=|\mathbf A_L|+s.
\]
Applying \emph{(assoc2)} gives
\[
((\mathbf 2^{\kappa_1}\circ_2 \mathbf A_R)\circ_1 \mathbf A_L)\circ_{u_a(\kappa_r)+1}\mathbf 2^{\kappa_r}
=
\bigl((\mathbf 2^{\kappa_1}\circ_2 \mathbf A_R)\circ_{s+1}\mathbf 2^{\kappa_r}\bigr)\circ_1 \mathbf A_L.
\]
Now \(\mathbf A_R =_{\mathcal I} f(a'_R)\), so
\[
|\mathbf A_R|=|f(a'_R)|=\operatorname{len}(a'_R)+1=\operatorname{len}(a_R).
\]
Hence
\[
2\le s+1<2+|\mathbf A_R|,
\]
and \emph{(assoc1)} yields
\[
(\mathbf 2^{\kappa_1}\circ_2 \mathbf A_R)\circ_{s+1}\mathbf 2^{\kappa_r}
=
\mathbf 2^{\kappa_1}\circ_2\bigl(\mathbf A_R\circ_s \mathbf 2^{\kappa_r}\bigr).
\]
Set
\[
\mathbf B_L:=\mathbf A_L,
\qquad
\mathbf B_R:=\mathbf A_R\circ_s \mathbf 2^{\kappa_r}.
\]
Since \(\mathbf A_R =_{\mathcal I} f(a'_R)\), it follows from the definition of \(f(a_R)\) that
\[
\mathbf B_R =_{\mathcal I} f(a_R).
\]
Therefore
\[
f(a)=_{\mathcal I}(\mathbf 2^{\kappa_1}\circ_2 \mathbf B_R)\circ_1 \mathbf B_L.
\]
\end{proof}

\begin{theorem}\label{thm:psi-via-f}
For every \(n\ge 1\) and every permutation \(\pi\in\Sym_n\), $\psi(\pi)=\eval(f(\pi))$.
\end{theorem}

\begin{proof}
We argue by induction on \(n\). If \(n=1\), then \(f(1)=\mathbf 2^1\), and
\(\eval(f(1))\) is the $2-$corolla. This is exactly \(\psi(1)\). Assume now that the statement holds for all permutations of size \(<n\), and let $\pi=\pi_L\,n\,\pi_R\in\Sym_n$. Apply Lemma~\ref{lem:f-word-normalization} to the word \(a=\pi\). Since
\(\kappa_1=\max(\pi)=n\), the root generator is \(\mathbf 2^n\), while the left
and right branches are determined by the subwords \(\pi_L\) and \(\pi_R\).

If both \(\pi_L\) and \(\pi_R\) are nonempty, Lemma~\ref{lem:f-word-normalization}
gives
\[
f(\pi)=_{\mathcal I}(\mathbf 2^n\circ_2 \mathbf B_R)\circ_1 \mathbf B_L,
\]
where
\[
\mathbf B_L =_{\mathcal I} f(\pi_L),
\qquad
\mathbf B_R =_{\mathcal I} f(\pi_R).
\]
Applying \(\eval\) and using Proposition~\ref{prop:eval-congruence}, we obtain
\[
\eval(f(\pi))
=
\eval(f(\pi_L))\vee \eval(f(\pi_R)).
\]
Standardization preserves the relative order of a word. Since the insertion
positions in the decreasing encoding \(f\) depend only on this relative order,
the terms \(f(a)\) and \(f(\std(a))\) have the same shape and differ only in
their indices for every nonempty word \(a\) of distinct positive integers.
Applying \(\eval\) forgets those indices, and hence
\[
\eval(f(a))=\eval(f(\std(a))).
\]
Applying this to \(a=\pi_L\) and \(a=\pi_R\), we get
\[
\eval(f(\pi_L))=\eval(f(\std(\pi_L))),
\qquad
\eval(f(\pi_R))=\eval(f(\std(\pi_R))).
\]
By the inductive hypothesis,
\[
\eval(f(\std(\pi_L)))=\psi(\std(\pi_L)),
\qquad
\eval(f(\std(\pi_R)))=\psi(\std(\pi_R)).
\]
Therefore
\[
\eval(f(\pi))
=
\psi(\std(\pi_L))\vee \psi(\std(\pi_R))
=
\psi(\pi).
\]

If \(\pi_L=\emptyset\), then Lemma~\ref{lem:f-word-normalization} gives
\[
f(\pi)=_{\mathcal I}\mathbf 2^n\circ_2 \mathbf B_R,
\qquad
\mathbf B_R=_{\mathcal I}f(\pi_R).
\]
Hence
\[
\eval(f(\pi))
=
\psi(\emptyset)\vee \eval(f(\pi_R)).
\]
Since \(\pi_R\) is nonempty in this case, standardization and the inductive
hypothesis give
\[
\eval(f(\pi_R))
=
\eval(f(\std(\pi_R)))
=
\psi(\std(\pi_R)).
\]
Thus
\[
\eval(f(\pi))
=
\psi(\emptyset)\vee\psi(\std(\pi_R))
=
\psi(\pi).
\]

The case \(\pi_R=\emptyset\) is analogous. In that case
\[
f(\pi)=_{\mathcal I}\mathbf 2^n\circ_1 \mathbf B_L,
\qquad
\mathbf B_L=_{\mathcal I}f(\pi_L),
\]
and hence
\[
\eval(f(\pi))
=
\psi(\std(\pi_L))\vee\psi(\emptyset)
=
\psi(\pi).
\]

This completes the induction.
\end{proof}

\begin{corollary}\label{cor:phi-psi-inverse}
For every \(n\ge 0\) and every permutation \(\pi\in\Sym_n\), $\phi(\pi)=\psi(\pi^{-1})$.
\end{corollary}

\begin{proof}
We argue by induction on \(n\). If \(n=0\), then \(\pi=\emptyset\), and both \(\phi(\emptyset)\) and
\(\psi(\emptyset^{-1})=\psi(\emptyset)\) are equal to \(\ast\in\Y_0\).
If \(n=1\), the claim is immediate. Assume now that the statement holds for all permutations of size \(<n\), and let $\pi=\pi_1\cdots \pi_n\in\Sym_n$. Set
\[
\sigma=\pi^{-1}\in\Sym_n.
\]
Since \(\sigma_{\pi_n}=n\), the recursive definition of \(\psi\) gives
\[
\psi(\sigma)=\psi(\std(\sigma_L))\vee \psi(\std(\sigma_R)),
\]
where
\[
\sigma_L=\sigma_1\cdots \sigma_{\pi_n-1},
\qquad
\sigma_R=\sigma_{\pi_n+1}\cdots \sigma_n.
\]
Now let
\[
\alpha=\pi^{<\pi_n},
\qquad
\beta=\pi^{>\pi_n}.
\]
The word \(\sigma_L\) consists of the positions, in \(\pi\), of the letters of
\(\alpha\), listed in increasing order of their values. Standardizing
\(\sigma_L\) replaces these original positions by their relative positions
among the entries of \(\alpha\). This is exactly the inverse of the
standardized word \(\std(\alpha)\). Therefore
\[
\std(\sigma_L)={\bigl(\std(\alpha)\bigr)}^{-1}.
\]
The same argument applied to the entries greater than \(\pi_n\) gives
\[
\std(\sigma_R)={\bigl(\std(\beta)\bigr)}^{-1}.
\]
By the inductive hypothesis,
\[
\psi(\std(\sigma_L))
=
\phi\bigl({\std(\sigma_L)}^{-1}\bigr)
=
\phi(\std(\alpha)),
\]
and likewise
\[
\psi(\std(\sigma_R))
=
\phi(\std(\beta)).
\]
Hence
\[
\psi(\pi^{-1})
=
\phi(\std(\pi^{<\pi_n}))\vee \phi(\std(\pi^{>\pi_n})).
\]
By the recursive definition of \(\varphi\) and the equality $\phi=\varphi$, we also have
\[
\phi(\pi)=\phi(\std(\pi^{<\pi_n}))\vee \phi(\std(\pi^{>\pi_n})).
\]
Therefore $\psi(\pi^{-1})=\phi(\pi)$, as required.
\end{proof}
\section{Order Preservation}

We show that Tonks' map \(\phi:\Sym_n\to\Y_n\) is order-preserving from the weak
Bruhat order to the Tamari order. Since the weak Bruhat order is the reflexive
transitive closure of its covering relation, it is enough to consider a cover
\[
\pi=\alpha\,u\,v\,\beta,
\qquad
\rho=\alpha\,v\,u\,\beta,
\qquad
u<v,
\]
so that \(\pi\lessdot_B\rho\). Set
\[
\tau=w(\beta),
\qquad
\gamma=w(\alpha).
\]
Then
\[
w(\pi)=\tau\,v\,u\,\gamma,
\qquad
w(\rho)=\tau\,u\,v\,\gamma.
\]
The local difference between \(w(\pi)\) and \(w(\rho)\) lies in the adjacent
block \(vu\) versus \(uv\), after the common prefix \(\tau\). Since the
insertion indices in \(h\) depend only on the letters already processed, the
prefix \(\tau\) determines the local comparison.

The case \(n=4\) is displayed in Figure~\ref{fig:tonks_map}. The red
intervals in the weak Bruhat order are precisely those collapsed by \(\phi\),
and the resulting quotient is the \(1\)-skeleton of the associahedron
\(\Asc^5\), with vertices indexed by the elements of \(\Y_4\) and ordered by the Tamari
order.

\begin{figure}[htbp]
     \centering
     \begin{subfigure}[t]{0.45\textwidth}
         \centering
         \includegraphics[width=\textwidth]{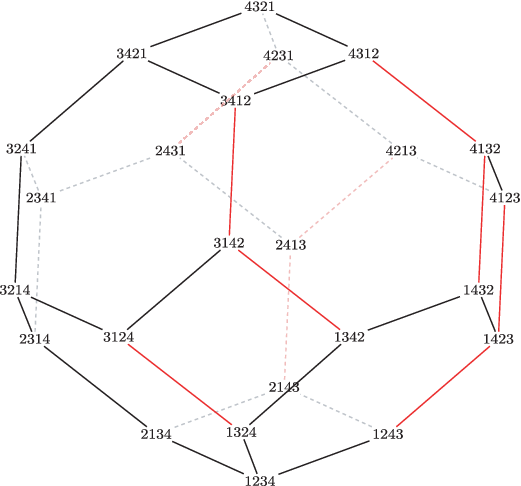}
         \caption{Weak Bruhat order on \(\Sym_4\).}\label{fig:bruhat}
     \end{subfigure}
     \hfill
     \begin{subfigure}[t]{0.45\textwidth}
         \centering
         \includegraphics[width=\textwidth]{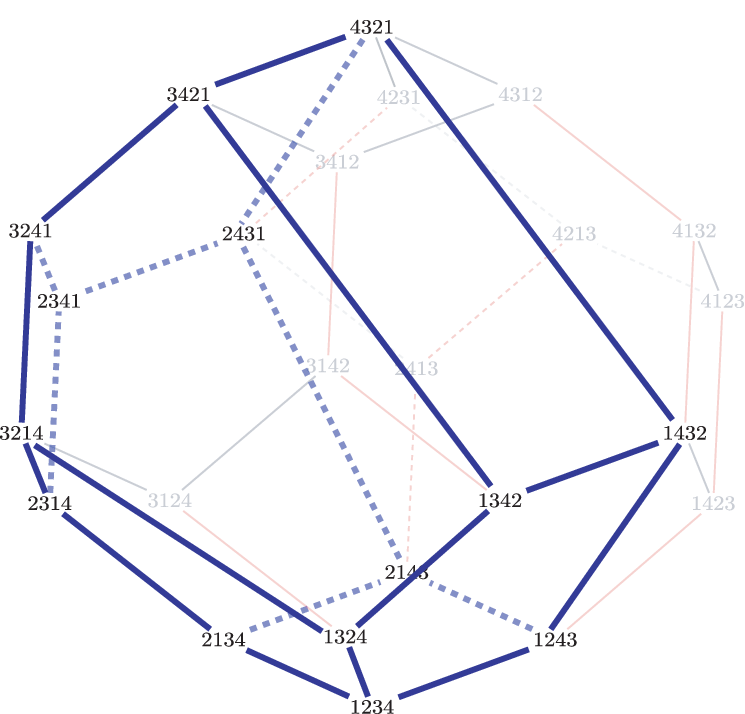}
         \caption{Induced Tamari order on the quotient \(\Sym_4/\sim\).}\label{fig:collapse}
     \end{subfigure}
        \caption{Weak Bruhat order on \(\Sym_4\) and the corresponding $1-$skeleton of the associahedron \(\Asc^5\).}\label{fig:tonks_map}
\end{figure}

\begin{lemma}\label{lem:local-indices}
Let \(\tau\) be a word of distinct integers, and let \(u<v\) be integers not occurring
in \(\tau\). Let
\[
n=k(v;\tau)
\qquad
\text{and}
\qquad
m=k(u;\tau v),
\]
where
\[
k(a;\sigma)=1+\bigl|\{x \text{ occurring in }\sigma : x<a\}\bigr|
\]
is the insertion index used in the definition of \(h\). Then
\[
n-m=\bigl|\{x \text{ occurring in }\tau : u<x<v\}\bigr|.
\]
\end{lemma}

%\begin{proof}
%By definition,
%\[
%m=1+\bigl|\{x \text{ occurring in }\tau : x<u\}\bigr|,
%\]
%since \(u<v\), and
%\[
%n=1+\bigl|\{x \text{ occurring in }\tau : x<v\}\bigr|.
%\]
%Therefore
%\[
%n-m=\bigl|\{x \text{ occurring in }\tau : u<x<v\}\bigr|.
%\]
%The three claims follow immediately.
%\end{proof}

\begin{lemma}\label{lem:independent-assoc2}
Let
\[
\pi=\alpha\,u\,v\,\beta,
\qquad
\rho=\alpha\,v\,u\,\beta,
\qquad
u<v,
\]
and set \(\tau=w(\beta)\). If \((u,v)\) is Tonks-independent in \(\pi\), then $h(w(\pi))=_{\mathcal I} h(w(\rho))$. Consequently, $\phi(\pi)=\phi(\rho)$.
\end{lemma}

\begin{proof}
Write
\[
w(\pi)=\tau\,v\,u\,\gamma,
\qquad
w(\rho)=\tau\,u\,v\,\gamma,
\]
where \(\gamma=w(\alpha)\). Since \((u,v)\) is Tonks-independent in \(\pi\),
there is an entry of \(\beta\) strictly between \(u\) and \(v\). Hence
\(\beta\) is nonempty, and therefore \(\tau=w(\beta)\) is nonempty. Let
\(\mathbf A=h(\tau)\).

Set
\[
n=k(v;\tau),
\qquad
m=k(u;\tau v).
\]
Since \(u<v\), we also have
\[
k(u;\tau)=k(u;\tau v)=m.
\]
Since \(u<v\), inserting \(u\) before \(v\) increases the insertion index of \(v\)
by \(1\), so
\[
k(v;\tau u)=k(v;\tau)+1=n+1.
\]
Therefore
\[
h(\tau v u)=(\mathbf A\circ_n \mathbf 2^v)\circ_m \mathbf 2^u,
\qquad
h(\tau u v)=(\mathbf A\circ_m \mathbf 2^u)\circ_{n+1}\mathbf 2^v.
\]
Since \((u,v)\) is Tonks-independent in \(\pi\), there exists \(x\in\beta\) with
\(u<x<v\). Equivalently, there exists \(x\in\tau\) with \(u<x<v\). By
Lemma~\ref{lem:local-indices}, we have \(m<n\). Hence
\[
m+2\le n+1,
\]
because \(|\mathbf 2^u|=2\). We may therefore apply \emph{(assoc2)} to the term
\(h(\tau u v)\), obtaining
\[
(\mathbf A\circ_m \mathbf 2^u)\circ_{n+1}\mathbf 2^v
=
(\mathbf A\circ_n \mathbf 2^v)\circ_m \mathbf 2^u.
\]
Thus
\[
h(\tau u v)=_{\mathcal I} h(\tau v u).
\]
The remaining letters, namely the letters of \(\gamma\), are then inserted in
the same order on both sides. At each step the two partial words contain the
same set of letters, so the next insertion has the same index in both terms.
Since \(=_{\mathcal I}\) is a congruence, it follows that
\[
h(w(\rho))=_{\mathcal I} h(w(\pi)).
\]
Applying \(\eval\) and Corollary~\ref{cor:phi-via-h}, we conclude
\[
\phi(\rho)=g(\rho)=g(\pi)=\phi(\pi).
\]
\end{proof}

\begin{lemma}\label{lem:dependent-rotation}
Let
\[
\pi=\alpha\,u\,v\,\beta,
\qquad
\rho=\alpha\,v\,u\,\beta,
\qquad
u<v,
\]
and set \(\tau=w(\beta)\). If \((u,v)\) is not Tonks-independent in \(\pi\), then
\[
\phi(\pi)<_T \phi(\rho).
\]
\end{lemma}

\begin{proof}
Write
\[
w(\pi)=\tau\,v\,u\,\gamma,
\qquad
w(\rho)=\tau\,u\,v\,\gamma,
\]
where \(\gamma=w(\alpha)\).

We first treat the case \(\tau\neq\emptyset\). Let $\mathbf A=h(\tau)$ and set
\[
n=k(v;\tau),
\qquad
m=k(u;\tau v).
\]
Since \((u,v)\) is not Tonks-independent in \(\pi\), there is no \(x\in\beta\) with
\(u<x<v\), hence no \(x\in\tau\) with \(u<x<v\). By
Lemma~\ref{lem:local-indices}, we have \(m=n\).

Since \(u<v\), we also have
\[
k(u;\tau)=k(u;\tau v)=m=n.
\]
Therefore
\[
h(\tau v u)=(\mathbf A\circ_n \mathbf 2^v)\circ_n \mathbf 2^u.
\]
By \emph{(assoc1)},
\[
(\mathbf A\circ_n \mathbf 2^v)\circ_n \mathbf 2^u
=
\mathbf A\circ_n(\mathbf 2^v\circ_1 \mathbf 2^u).
\]
Moreover, since \(u<v\),
\[
k(v;\tau u)=k(v;\tau)+1=n+1.
\]
Hence
\[
h(\tau u v)=(\mathbf A\circ_n \mathbf 2^u)\circ_{n+1}\mathbf 2^v.
\]
Applying \emph{(assoc1)} again, we obtain
\[
(\mathbf A\circ_n \mathbf 2^u)\circ_{n+1}\mathbf 2^v
=
\mathbf A\circ_n(\mathbf 2^u\circ_2 \mathbf 2^v).
\]
Now
\[
\eval(\mathbf 2^v\circ_1 \mathbf 2^u)
\qquad\text{and}\qquad
\eval(\mathbf 2^u\circ_2 \mathbf 2^v)
\]
differ by exactly one right rotation, so
\[
\eval(h(\tau v u))<_T \eval(h(\tau u v)).
\]
If \(\tau=\emptyset\), then
\[
h(vu)=\mathbf 2^v\circ_1 \mathbf 2^u,
\qquad
h(uv)=\mathbf 2^u\circ_2 \mathbf 2^v,
\]
and therefore again
\[
\eval(h(vu))<_T \eval(h(uv)).
\]

It remains to pass from the prefixes \(\tau vu\) and \(\tau uv\) to the
words \(w(\pi)\) and \(w(\rho)\). Let the letters of \(\gamma\) be inserted one
after another. At each step the two partial words have the same set of letters,
and the insertion index for the next letter depends only on this set. Hence
the next insertion has the same index on both sides.

Tamari order is closed under substitution into a fixed leaf context: if
\(S<_T T\), then replacing the same leaf of any binary tree by \(S\) and by
\(T\) again gives a strict Tamari inequality. The rotation witnessing
\(S<_T T\) occurs inside the substituted subtree, while the surrounding context
is unchanged. Applying this observation successively to the common insertions
of the letters of \(\gamma\), we obtain
\[
g(\pi)<_T g(\rho).
\]
Using Corollary~\ref{cor:phi-via-h}, we conclude
\[
\phi(\pi)<_T \phi(\rho).
\]
\end{proof}

\begin{theorem}\label{thm:order-preservation}
The map $\phi:\Sym_n\to\Y_n$ is order-preserving from the weak Bruhat order on \(\Sym_n\) to the Tamari order on \(\Y_n\).
\end{theorem}

\begin{proof}
It is enough to consider a cover
\[
\pi\lessdot_B \rho,
\]
so that
\[
\pi=\alpha\,u\,v\,\beta,
\qquad
\rho=\alpha\,v\,u\,\beta,
\qquad
u<v.
\]
If \((u,v)\) is Tonks-independent, then Lemma~\ref{lem:independent-assoc2} gives $\phi(\pi)=\phi(\rho)$. If \((u,v)\) is not Tonks-independent, then Lemma~\ref{lem:dependent-rotation} gives $\phi(\pi)<_T \phi(\rho)$. In either case, $\phi(\pi)\le_T \phi(\rho)$. Since the weak Bruhat order is the reflexive transitive closure of its covering relation, it follows that \(\phi\) is order-preserving.
\end{proof}

\section{Conclusion}

We have given a syntactic account of the two maps from permutations to binary
trees considered in this paper. Tonks' map \(\phi\) is obtained by evaluating the head-insertion encoding \(h\) on the reversed permutation, while the Loday--Ronco map \(\psi\) is
obtained by evaluating the decreasing encoding \(f\) on the permutation itself. Both constructions therefore reside in the same equational calculus
\(\mathcal I\) for the free non-symmetric, non-unital operad on one binary
generator.

The comparison also explains the relation between the two recursive descriptions.
For Tonks' map, the relevant decomposition is governed by the first letter of
the reversed word, equivalently by the last letter of the original permutation.
For the Loday--Ronco map, the decreasing encoding is normalized by splitting at
the maximal entry.

In particular, the order-preservation property of Tonks' map admits a local
operadic explanation. In the independent case, the identification of adjacent
swaps is governed by instances of \emph{(assoc2)}. In the dependent case, the
strict Tamari increase is exposed by the corresponding local rewriting analysis,
which yields the relevant rotation on the tree side.

It remains to ask whether similar normal-form methods apply to other nestohedra where canonical
projections and quotient constructions also occur.

\end{document}